\documentclass[12pt]{article}

\usepackage{amsmath}
\usepackage{latexsym}
\usepackage{amsfonts}
\usepackage{amssymb}
\usepackage{amscd}
\usepackage{theorem}
\usepackage{a4wide}
\usepackage{color}

\newtheorem{theorem}{Theorem}
\newtheorem{definition}[theorem]{Definition}
\newtheorem{lemma}[theorem]{Lemma}
\newtheorem{proposition}[theorem]{Proposition}
\newtheorem{corollary}[theorem]{Corollary}

{\theorembodyfont{\rmfamily}
	\newtheorem{example}[theorem]{Example}

	\newtheorem{remark}[theorem]{Remark}

}

\newenvironment{proof}{    
	\noindent
	\textbf{Proof.}}{
	\hfill $\Box$
	\vspace{3mm}
}

\numberwithin{equation}{section}

\newcommand{\N}{\mathbb{N}} 
\newcommand{\R}{\mathbb{R}} 
\newcommand{\C}{\mathbb{C}} 

\title{On the spectrum of isomorphisms defined on the space of smooth functions which are flat at 0}
\author{ Enrique Jord\'a \footnote{Instituto Universitario de Matem\'atica Pura y Aplicada IUMPA, Universitat Polit\`ecnica de Val\`encia, Camino de Vera, s/n, E-46022 Valencia, Spain.}}

\begin{document}

	\maketitle
	
	\begin{abstract}
In this note we study the spectrum and the Waelbroeck spectrum of the derivative operator composed with isomorphic multiplication operators defined in the space of smooth functions in $[0,1]$ which are flat at 0. 

\end{abstract}

\section{Introduction and preliminaries}

\subsection{Introduction}

The spectrum of a continuous linear operator $T$ defined on a locally convex space $X$ is defined in an analogous way as in the case when $X$ is a Banach space. Given $T\in \mathcal{L}(X)$, (here $\mathcal{L}(X)$ stands for the continuous linear operators on $X$), the resolvent  of $T$,  denoted by $\varrho(T)$, is defined as the subset of $\C$ formed by those $\lambda$ such that  $\lambda I-T$  admits a continuous linear inverse $(\lambda I- T)^{-1}$. For $\lambda\in \varrho(T)$ we denote, as usual, $R(\lambda,T)=(\lambda I-T)^{-1}\in L(X)$. When $X$ is a Fr\'echet space, $\lambda I-T$ is an isomorphism if and only if $\lambda I-T$ is bijective. The spectrum of $T$ is defined as $\sigma(T):=\C\setminus \varrho(T)$. The point spectrum of $T$ is defined as $\{\lambda\in\C: T(x)=\lambda x \text{ for some } x\neq 0\}$. Due to the open mapping theorem, when $X$ is a Fr\'echet space, if $\lambda\in \C\setminus \sigma_p(T)$, then $\lambda\in \varrho(T)$ if and only if $\lambda I-T$ is surjective. 
Contrary to what happens on the Banach spaces, the spectrum of an operator defined on a Fr\'echet space could be empty, or unbounded.  Several authors consider the Waelbroeck spectrum of the operator, as a natural way in order to get holomorphy in the resolvent map. The Waelbroeck resolvent $\varrho^*(T)$ is defined as the subset formed for those $\lambda\in \varrho(T)$ such that there is a neighbourhood $V_\lambda$ of $\lambda$ contained in $\varrho(T)$ such that $\{R(\lambda,T):\ \lambda\in V_\lambda\}$ is an equicontinuous subset of $\mathcal{L}(X)$. The Waelbroeck spectrum $\sigma^*(T)$ of $T$ is defined as $\C\setminus \varrho^*(T)$ (see \cite{Vasilescu}). From the definition it follows immediately $\overline{\sigma(T)}\subseteq \sigma^*(T)$.  The inclusion can be strict, as it can be checked in \cite[Remark 3.5 (vi)]{abr1}. The example, stated without proof in \cite[Example 2]{maeda}, is the Volterra  operator in the space 

$$C_0^\infty([0,1]):=\{f\in C^\infty([0,1]):\ f^{(k)}(0)=0\ \text{ for all }k\in\N_0\}.$$

\noindent  This space is endowed with its natural topology, which is generated by the norms
$\|f\|_n:=\sup_{1\leq j \leq n} \{|f^{(j)}(x)|:\ x\in [0,1],\ 0\leq j\leq n\}.$ These family of norms makes $C_0^\infty([0,1])$ a Fr\'echet nuclear space. More explicitly, \cite[Remark 3.5 (vi)]{abr1} can be stated  as follows:

\begin{proposition}
\label{isomorphism}
The following operators are surjective  isomorphisms in $C_0^\infty([0,1])$
\begin{itemize}
\item[(a)]The derivative operator $D:C_0^\infty[0,1]\to C_0^\infty[0,1]$, $f\mapsto f'$ is an isomorphism which satisfies $\sigma(D)=\sigma^*(D)=\emptyset$.
\item[(b)] The inverse of $D$ is the Volterra operator $V:C_0^\infty([0,1])\to C_0^\infty([0,1])$, $f\mapsto V(f)(x):=\int_0^xf(t)dt$, $x\in [0,1]$, which satisfies $\sigma(V)=\emptyset$, $\sigma^*(V)=\{0\}$.
\end{itemize}
\end{proposition}

The study of the spectrum of operators defined on Fr\'echet spaces or more general locally convex spaces has been an object of research in the last years, see e.g. \cite{abr1,abr2,abr6,abr7,b2,fgj1,fgj2,k,r}.

Several of the aforementioned references are devoted to the study of the Ces\`{a}ro operator in spaces of functions. Our main motivation is \cite{abr2,abr25}, where Albanese, Bonet and Ricker showed that the Ces\`{a}ro operator $C$ defined on $C^\infty(\R_+)$ satisfies $\sigma(C)=\sigma_p(C)=\{1/n:\ n\in\N\}$ and $\sigma^*(C)=\overline{\sigma(C)}$. We study a class of operators which includes the Ces\`{a}ro operator defined in the space $C_0^\infty([0,1])$, whose spectrum has been recently characterized by Albanese in \cite{a}.

\subsection{Spectrum of operators on locally convex spaces}   
\label{lcs}
 In this note we are concerned in this note with spectra of isomorphisms on Fr\'echet spaces. In the next proposition we include first a basic result which compares spectra and Waelbroeck spectra of $T$ and $T^{-1}$ defined on a  locally convex space $X$. It is a particular case of \cite[Theorem 1.1]{abr2}, due to Albanese, Bonet and Ricker.
\begin{proposition}
\label{inverse}
Let $X$ be a  locally convex space and $T\in \mathcal{L}(X)$ be an isomorphism. $\sigma(T^{-1})=\{\lambda^{-1}: \ \lambda\in \sigma(T)\}$ and $\sigma^*(T^{-1})\setminus\{0\}=\{\lambda^{-1}: \ \lambda\in \sigma^*(T)\setminus\{0\}\}$.
\end{proposition}
As a consequence of Proposition \ref{inverse},  if $T$ is an isomorphism and $\lambda\neq 0$ is an accumulation point in $\sigma^*(T)\setminus\{0\}$ if and only if $\lambda^{-1}$ is an accumulation point of $\sigma^*(T^{-1})\setminus\{0\}$. For $\lambda=0$ we see below that nothing  can be asserted. When $X$ is a Banach space and $T$ is an isomorphism on $X$ then $0$ is neither in the (Waelbroeck) spectrum of $T$ nor in that of $T^{-1}$. In the case of Fr\'echet spaces we see that $0$ can appear in the Waelbroeck spectrum of an isomorphism and in that of its inverse, and that when it appears it can be both,    an isolated  point or  an accumulation point.  In the next example, we consider the space $\omega=\C^{\N}$ of sequences of complex numbers endowed with the product topology. The proof  relies on the fact that, if $X,Y$ are locally convex spaces, $T\in \mathcal{L}(X)$ and $S\in\mathcal{L}(Y)$, and we consider the direct sum $T\oplus S\in \mathcal{L}(X\oplus Y)$, then $\sigma_p(T\oplus S)=\sigma_p(T)\cup \sigma_p(S)$, $\sigma(T\oplus S)=\sigma(T)\cup \sigma(S)$ and $\sigma^*(T\oplus S)=\sigma^*(T)\cup \sigma^*(S)$.

\begin{example}
Let $T:\omega\to \omega$, $(x_n)\mapsto (nx_n)$. Then $T$ is an isomorphism, $T^{-1}:\omega\to \omega, (x_n)\mapsto (\frac1n x_n)$, $\sigma(T)=\sigma_p(T)=\sigma^{*}(T)=\N$ and 
$\sigma(T^{-1})=\sigma_p(T^{-1})=\{\frac1n:\ n\in\N\}$ and $\sigma^*(T^{-1})=\overline{\sigma(T^{-1})}=\sigma(T^{-1})\cup \{0\}$.

 The operators $D$ and $V$ defined on $C_0^\infty([0,1])$ are the same as in Proposition \ref{isomorphism}.
\begin{itemize}
\item[(a)] $S:=T\oplus T^{-1}\in\mathcal{L}(\omega\oplus \omega)$  satisfies $\sigma(S)=\sigma(S^{-1})=\sigma_p (S)=\sigma_p(S^{-1})=\N\cup \{\frac1n:\ n\in\N\}$ and  $\sigma^*(S)=\sigma(S^{-1})=\overline{\sigma(S)}=\sigma(S)\cup\{0\}$. 
\item[(b)] $S:=D\oplus V\in\mathcal{L}(C_0^\infty([0,1])\oplus C_0^\infty([0,1]))$ satisfies $\sigma(S)=\sigma(S^{-1})=\emptyset$ and $\sigma^*(S)=\sigma^*(S^{-1})=\{0\}$.

\item[(c)]  $S:=D\oplus T\in \mathcal{L}(C_0^\infty([0,1])\oplus \omega)$ satisfies $\sigma(S)=\sigma^*(S)=\N$, $\sigma(S^{-1})=\{\frac1n:\ n\in\N\}$ and $\sigma^*(S^{-1})=\overline{\sigma(S^{-1})}=\sigma(S^{-1})\cup\{0\}$.

\item[(d)] $S:=V \oplus  T\in \mathcal{L}(C_0^\infty([0,1])\oplus \omega)$ satisfies $\sigma(S)=\N$, $\sigma^*(S)=\N_0$, $\sigma(S^{-1})=\{\frac1n:\ n\in\N\}$ and $\sigma^*(S^{-1})=\overline{\sigma(S^{-1})}=\sigma(S^{-1})\cup\{0\}$.

\end{itemize}

\end{example}

The next result is stated for Banach algebras in \cite[Exercise 7.3.7]{Conway}
\begin{proposition}
\label{com}
Let $X$ be a locally convex space and let $A,B\in \mathcal{L}(X)$. Then $\sigma(AB)\cup \{0\}=\sigma(BA)\cup\{0\}$. 
\end{proposition}
\label{esp}
\begin{proof}
For $\lambda\in \varrho(AB)\setminus \{0\}$,  set $T:=\lambda^{-1}I+\lambda^{-1}B(\lambda I-AB)^{-1}A\in\mathcal{L}(X)$. A direct computation shows $T=(\lambda I-BA)^{-1}$.

\end{proof}

\begin{proposition}
\label{esp2}
Let $X$ be a locally convex space and let $A,B\in \mathcal{L}(X)$, $B$ being an isomorphism. Then $\sigma_p(AB)=\sigma_p(BA)$, $\sigma(AB)=\sigma(BA)$ and $\sigma^*(AB)=\sigma*(BA)$.
\end{proposition}

\begin{proof}
If $\lambda\in\sigma_p(AB)$ and $x\in X\setminus\{0\}$ satisfies $AB x=\lambda x$, then $BA(Bx)=\lambda Bx$ and $Bx\neq 0$, and hence $\lambda\in \sigma_p(BA)$. Conversely, if $\lambda\in\sigma_p(BA)$ and $x\in X\setminus\{0\}$ satisfies $BA x=\lambda x$, then there is $y\in X$ such that $By=x$. Then $BABy=\lambda By$, and the injectivity of $B$ yields $ABy=\lambda y$, and consequently $\lambda\in \sigma_p(AB)$.

Let assume now $\lambda\in \varrho(AB)$. We set $T_{\lambda}:=B(\lambda I-AB)^{-1}B^{-1}$. It can be checked $T_\lambda (\lambda I-BA)=(\lambda I-BA)T_\lambda=I$. Hence $\sigma(BA)\subseteq\sigma(AB)$.  Conversely, if $\lambda\in \varrho(BA)$, one can check that $Q_\lambda:=B^{-1}(\lambda I-BA)^{-1}B$ is the inverse of $(\lambda I-AB)$. Thus we have $\sigma(AB)=\sigma(BA)$. Moreover,  for any  compact set $K\subseteq\varrho(AB)$ we have
$$\{(\lambda I-BA)^{-1}: \ \lambda\in K\}=\{B(\lambda I-AB)^{-1}B^{-1}: \ \lambda\in K\}.$$

We conclude $\{(\lambda I-BA)^{-1}: \ \lambda\in K\}$ is equicontinuous if and only if  $\{(\lambda I-AB)^{-1}: \ \lambda\in K\}$ is so, i.e. if and only if $\sigma^*(AB)=\sigma^*(BA)$.
\end{proof}

\subsection{Representation of $C_0^\infty([0,1])$}
The space $C_0^\infty([0,1])$ is well known to be isomorphic to the space $s$ of rapidly decreasing sequences. Bargetz has obtained in $\cite{bargetz2}$  an explicit isomorphism, which it is used in \cite{bargetz} to obtain explicit representations as sequence spaces of important spaces of smooth functions appearing in functional analysis. We study in this note a wide class of isomorphisms defined on this space containing the differentiation operator, the Volterra operator and also the Ces\`{a}ro operator.
 To do this, we need a representation of $C_0^\infty([0,1])$, as the one sided Schwartz space of rapidly decreasing smooth functions $S(\R^+)$. There is a natural representation for the one unit translate of this space 

$$S([1,\infty)):=\{f\in C^\infty([1,\infty)): \ \lim_{x\to\infty} x^nf^{(j)}(x)=0 \text{ for all }j,n\in\N_0\}.$$ 

\noindent To get such representation, we need the well known Fa\`{a} di Bruno formula, which we state below. Let $x\in\R$: if $g$ is $C^j$, i.e $f$ admits continuous derivatives up to order $j$, at $x$ and $f$ is $C^j$ at $f(x)$ then

$$(f\circ g)^{(j)}(x)=\sum_{i=1}^{j} f^{(i)}(g(x))B_{j,i}(g'(x),g''(x),\ldots,g^{(j-i+1)}(x)),$$

\noindent  where $B_{j,i}$ are the Bell polynomials

\begin{equation}
\label{bell1}
B_{j,i}(x_1,x_2,\ldots,x_{j-i+1})=\sum \frac{j!}{i_1!i_2!\cdots i_{j-i+1}!}\left(\frac{x_1}{1!}\right)^{i_1}\cdots \left(\frac{x_{j-i+1}}{(j-i+1)!}\right)^{i_{j-i+1}},
\end{equation}

\noindent $i_1+\cdots i_{j-i+1}=i$, $i_1+2i_2+\cdots+(j-i+1)i_{j-i+1}=j$.

\begin{remark}
\label{bell}
\begin{itemize}
\item[(a)]
From \eqref{bell1}, it follows that $|B_{j,i}(x_1,x_2,\ldots,x_{j-i+1})|\leq B_{j,i}(y_1,y_2,\ldots,y_{j-i+1})$ whenever  $|x_l|\leq |y_l|$, $1\leq l\leq j-i+1$.

\item[(b)]Let $i\leq j$ and let $f_l(x)$ functions defined on  $ (0,1]$  such that,   there exists  $t(i)\in \R$ such that $|f_l(x)|\leq x^{t(i)}$ for each $x\in (0,1]$ and for each $1\leq l\leq j-i+1$,. Then   there exists $M>0$, $t\in\R$ such that  $|B_{j,i}(f_1(x),f_2(x),\ldots,f_{j-i+1}(x))|\leq Mx^t$
 \end{itemize}
\end{remark}

\begin{proposition}
\label{description}
$C_0^\infty([0,1])=\{f\in C^{\infty}([0,1]):\  f^{(j)}(x)=o(x^n)\text{ as }x\text{ approaches to } 0$  $\text{ for all }j,n\in \N_0\}$.
\end{proposition}

\begin{proof}
. Let $f\in C_0^\infty([0,1])$. The mean value theorem implies $|f(x)|=|f'(t)|x$ for some $t\in (0,x)$. A reiteration of the argument produces $|f(x)|\leq \sup_{t\in[0,1]}|f^{(n)}(t)|x^n$ for all $n\in\N_0$. The condition $f^{(j)}(x)=o(x^n)$ as $x$ approaches to $0$ follows  from the fact that  $f^{(j)}\in C_0^\infty([0,1])$ for all $j\in\N$. The other inclusion is trivial.
\end{proof}
 
 From Proposition \ref{description} and Leibnitz's formula, it follows immediately the following corollary.

\begin{corollary}
\label{monomials}
 and Given $f\in C_0^\infty([0,1])$ then $g(x):=x^tf\in C_0^\infty([0,1])$ for any $t\in \R$ (by defining $g(0)=0$ when $t<0$).
\end{corollary}

\begin{theorem}
\label{representation}
The map $T:\ C_0^\infty([0,1])\to S([1,\infty))$,  $f \mapsto T(f)$, defined as  $T(f)(x)=\tilde{f}(x):=f(1/x)$, $x\in(0,1]$, is an isomorphism. 
\end{theorem}

\begin{proof}
For  $n\in\N$, by the Fa\`{a} di Bruno F\'ormula we have

\begin{equation}
\label{welldefined}
x^n\tilde{f}^{(j)}(x)=  x^{n}\sum_{i=1}^{j} f^{(i)}(1/x)B_{j,i}(-x^{-2},2x^{-3},\ldots, (-1)^{j-i+1}(j-i+1)!x^{-(j-i+2)}).
\end{equation}

By Remark \ref{bell}, we get $M>0$ ,$a\in\R$ such that, for all $1\leq i\leq j$ 

\begin{equation}
\label{welldefined2}
|x^{n}B_{j,i}(-x^{-2},\ldots, (-1)^{j-i+1}(j-i+1)!x^{-(j-i+2)})|\leq M x^{a}
\end{equation}

Let $t:=1/x$. From \eqref{welldefined} , \eqref{welldefined2} and Proposition \ref{description} we get

\begin{equation}
\label{welldefined3}
\lim_{x\to\infty} x^n|\tilde{f}^{(j)}(x)|\leq M \lim_{t\to 0^+}   \sum_{i=1}^{j}f^{(i)}(t)t^{-a}=0
\end{equation}

Then $T$ is well defined. The mapping $T $ is injective and continuous by the closed graph theorem, since it is obviously pointwise--pointwise continuous. We see that $T$ is also surjective. Given $f\in S([1,\infty))$, we make an abuse of notation to define $\tilde{g}(x)=g(1/x)$, $x\in (0,1]$, $\tilde{g}(0)=0$. For all $j,n\in \N$, a completely symmetric argument to that used for getting \eqref{welldefined3}, gives 

$$\lim_{x\to 0^+} x^{-n} \tilde{g}^{(j)}(x)=0.$$

\noindent  From Proposition \ref{description}, it follows $\tilde{g}\in C_0^\infty([0,1])$. We conclude from $T(\tilde{g})=\tilde{\tilde{g}}=g$.

\end{proof}

\section{Spectrum of multipliers on $C_0^\infty([0,1])$}

In view of Theorem \ref{representation}, the results given in this section are closely related to \cite[Proposition 3.3, Remark 3.5]{am}.
\begin{definition}
The space of multipliers of $C_0^\infty([0,1])$ is defined as 
$$\mathcal{M}:=\{\omega:(0,1]\to \C: \forall f\in C_0^\infty([0,1])\ \omega f\text{ can be extended to }0\text{ as a function in } C_0^\infty([0,1]) )\}.$$ 

\noindent  For $\omega\in\mathcal{M}$, we denote by $M_\omega:C_0^\infty([0,1])\to C_0^\infty([0,1]),\ f\mapsto \omega f$ the corresponding multiplication operator. 
\end{definition}
\begin{lemma}
\label{multipliers1}
A function $\omega:(0,1]\to\C$ satisfies $\omega \in \mathcal{M}$  if and only if $\omega\in C^\infty((0,1])$ and,  for each $j\in\N_0$, there is $n\in\N$ such that $\omega^{(j)}(x)=o(x^{-n})$ as $x$ approaches to 0. 
\end{lemma}

\begin{proof}
By the definition, it is immediate to show that $\omega\in C^{\infty}(0,1]$ whenever $\omega\in \mathcal{M}$. Theorem \ref{representation} yields that $\omega\in \mathcal{M}$ if and only if $\tilde{\omega}(x):=\omega(1/x)$ is a multiplier in $S([1,\infty))$.  By the standard proof characterizing the multipliers of $S(\R)$ (see \cite{grote}),  this is equivalent to $\tilde{\omega}\in C^\infty([1,\infty))$ and for all $j\in\N$ there is $n\in\N$ such that $\tilde{\omega}^{(j)}(x)=o(x^n)$ as $x$ goes to $\infty$. 
Let $k\in\N$ such that $\tilde{\omega}'(x)=o(x^k)$ as $x$ goes to $\infty$. This is equivalent to $\omega'(x)=o(x^{-k+2})$ as $x$ approaches $0$. Using Fa\`{a} di Bruno formula one gets inductively the statement.
\end{proof}

\begin{proposition}
\label{multipliers}

Let $\omega\in\mathcal{M}$. The multiplication operator $M_\omega:C_0^\infty([0,1])\to C_0^\infty([0,1])$ is an isomorphism if and only if $\omega(x)\neq 0$ for all $x\in (0,1]$   and $1/\omega\in\mathcal{M}$.

\end{proposition}
\begin{proof}
First we observe that if $\omega(x_0)=0$ for some $x_0\in (0,1]$ , then $M_\omega f(x_0)=0$ for all $f\in C_0^\infty([0,1])$. Hence $M_\omega$ is not surjective (observe, for instance $f(x)=e^{-1/x}\in C_0^\infty([0,1])$).  If $M_\omega$ is an isomorphism, then the inverse $T$ satisfies $M_\omega T(f)=\omega Tf=f$, hence $T(f)(x)=(1/\omega(x))f(x)$ for all $x\in (0,1]$. This means $1/\omega\in\mathcal{M}$ and $T=M_{1/\omega}$. The converse is trivial.
\end{proof}

\begin{corollary}
\label{multipliers2}
Let $\omega\in\mathcal{M}$, the multiplication operator $M_\omega:C_0^\infty([0,1])\to C_0^\infty([0,1])$ is an isomorphism if and only if $\omega(x)\neq 0$ for all $x\in (0,1]$   and there is $m\in\N$ such that $(1/\omega(x))=o(x^{-m})$ as $x$ approaches 0. In particular, for every $p\in\R$, if we define $\omega_p(x):=x^p$,  then $M_{\omega_p}$ is an isomorphism.
\end{corollary}

\begin{proof}
By Lemma \ref{multipliers1} and Proposition \ref{multipliers}, we only need to show the suficiency of the condition. 
 Assume that $\omega(x)\neq 0$ and there is $m\in\N$ such that $((1/\omega(x))=o(x^{-m})$ as $x$ approaches 0. We need to show that $1/\omega\in\mathcal{M}$. Let $j\in\N$. By applying Fa\`{a} di Bruno formula we get 

$$((\omega(x))^{-1})^{(j)}=  \sum_{i=1}^{j} (-1)^i i!(\omega(x))^{-i-1}B_{j,i}(\omega'(x),\cdots ,\omega^{(j-i+1)}(x)).$$

\noindent Therefore we apply the hypothesis to get $k\in\N$ such that  $(\omega(x)^{-1})^{(j)}=o(x^{-k})$ as $x$ approaches 0. The conclusion follows from Lemma \ref{multipliers1} and Remark \ref{bell}.    
\end{proof}

\begin{corollary}
\label{c14}
 If $\omega\in \mathcal{M}$  then the spectrum of $M_\omega$ is

$$\sigma(M_\omega)=\omega((0,1])\cup \{\lambda\notin \omega((0,1]):\  x^n/(\lambda-\omega) \text{ unbounded in }(0,1]\ \forall n\in\N \}.$$
\end{corollary}

\begin{proof}
It follows from Corollary \ref{multipliers2} applied to the multiplier $\lambda-\omega$, for any $\lambda\in\C\setminus \omega((0,1])$. 
\end{proof}

\begin{lemma}
\label{eq}
Let $(\omega_i)_{i\in I}\subseteq \mathcal{M}$. Assume that for each $j\in \N$ there is $M_j>0$ and $t(j)\in \R$  such that $|\omega_i^{(j)}(x)|\leq M_jx^{t(j)}$ for every $x\in (0,1]$, $i\in I$. Then the set of multiplication operators $(M_{\omega_i})_{i\in I}\subseteq \mathcal{L}(C_0^\infty([0,1]))$ is equicontinuous.
\end{lemma}
\begin{proof}
 By the Banach Steinhauss theorem, we only need to show that, for every $f\in C_0^\infty([0,1])$, the set
 $\{M_{\omega_i}(f):\ i\in I\}$ is bounded in $C_0^\infty([0,1])$. This happens when, for every  $k\in\N_0$, $\{(M_{\omega_i}f)^{(k)}(x):\ x\in [0,1]\}$ is bounded in $\C$. Let fix $k\in\N_0$.  Let $M=\max\{M_j:\ 0\leq j\leq k\}$, $t=\min\{t(j):\ 0\leq j\leq k\}$.
$$\max_{x\in [0,1]}|(\omega_if)^{(k}(x)|=\sup_{x\in (0,1]}\left |\sum_{j=0}^{k}\omega_i^{(j)}(x)f^{(k-j)}(x)\right|\leq M\sup_{x\in (0,1]}\sum_{j=0}^{k} x^t |f^{(k-j)}(x)|.$$

\noindent We conclude since $x^t\in \mathcal{M}$ by Corollary \ref{multipliers}.
\end{proof}

The following proposition is a direct consequence of Lemma \ref{eq}.

\begin{proposition}\label{eqex}
Let $K\subseteq \C$ be compact.  The multiplication operators $\{M_{h_\lambda}:\ \lambda\in K  \}$ form an equicontinuous subset of $\mathcal{L}(C^\infty_0([0,1]))$ in the following cases:

\begin{itemize}

\item[(a)] $h_\lambda(x):=f(\lambda)x^{\lambda}$, $f\in C(K)$.

\item[(b)] $h_\lambda(x):=f(\lambda)e^{g(\lambda) x^{t}}$, if  $t\geq 0$ and $f,g\in C(K)$.
\end{itemize}
\end{proposition}

\begin{proposition}
\label{w}
Let $\omega\in \mathcal{M}$.  The Waelbroeck spectrum of $M_\omega$ is $\sigma^*(M_\omega)=\overline{\sigma(M_\omega)}=\overline{\omega((0,1])}$.
\end{proposition}
\begin{proof}
We first observe  $\{\lambda:  x^n/(\lambda-\omega) \text{ unbounded in }(0,1] \ \forall n\in\N \}\subseteq \overline{\omega((0,1])}$. Actually, for any such $\lambda$ there must exist a sequence $(x_n)\subseteq (0,1]$ convergent to $0$ such that $\lim \omega(x_n)=\lambda$. Hence, by Corollary \ref{c14}, we have $\overline{\sigma(M_\omega)}=\overline{\omega((0,1])}$. We only need to show $\sigma^*(M_\omega)\subseteq \overline{\omega((0,1])}$. Let  $\lambda_0\in \C\setminus \overline{\omega((0,1])}$. We choose $r>0$ such that there exists $c>0$ satisfying
$$|\lambda-\omega(x)|>c\quad  \forall \lambda\in B(\lambda_0,r),\ x\in (0,1].$$

\noindent From Lemma \ref{multipliers1}, given $j\in\N_0$  there is $k\in\N$ and $C>0$ such that $|\omega^{(i)}(x)|\leq C x^{-k}$ for $1\leq i\leq j$. From this, Fa\`{a} di Bruno formula and Remark \ref{bell} we get $M>0$, $t\in \R$ such that, for every $\lambda\in B(\lambda_0,r)$, $x\in (0,1]$ we have 

$$|((\lambda-\omega(x))^{-1})^{(j)}|=  \left|\sum_{i=1}^{j} (\lambda-\omega(x))^{-i-1}B_{j,i}(\omega'(x)\cdots \omega^{(j-i+1)}(x))\right |\leq  M x^{t}.$$ 

\noindent Lemma \ref{eq}  gives the equicontinuity of $\{M_{(\lambda-\omega(x))^{-1}}:\ \lambda\in B(\lambda_0,r)\}$. Hence  $\lambda_0\in  \varrho^*(M_\omega)$.

\end{proof}

From Corollary \ref{c14} and Proposition \ref{w} we get the following:

\begin{example}
\label{iso2}
For $p\in \R$, let  $\omega_p(x)=x^p\in \mathcal{M}$.
\begin{itemize}

\item[(i)] If $p>0$ then $\sigma(M_{\omega_p})=(0,1]$ and $\sigma^*(M_{\omega_p})=[0,1]$
\item[(ii)] If $p<0$ then $\sigma(M_{\omega_p})= \sigma^*(M_{\omega_p})=[1,\infty)$.
\end{itemize}
\end{example}

\section{Spectrum of Ces\`{a}ro type operators on $C_0^\infty([0,1])$}

 In this section we study the spectrum of operators $VM_{\omega_p}$ and also $M_{\omega_p}V$, where $\omega_p=x^p$, $p\in\R$, $M_{\omega_p}$ is the multiplication operator and $V$ is the Volterra operator. These operators are isomorphisms in view of Proposition \ref{isomorphism} and Corollary \ref{multipliers2}. The relevant case $M_{\omega_{-1}}V$ gives the Ces\`{a}ro operator. By Proposition \ref{inverse}, the results will determine the spectrum of $DM_{\omega_p}$ and $M_{\omega_p}D$.

\begin{lemma}
\label{tecnic}
Let $g\in C_0^\infty([0,1])$, $q<0$,  and $c\in \C$ with $\text{Re}(c)>0$. Then $h_c(x):=\int_0^x e^{c(x^q-t^q)}g(t)dt\in C_0^\infty([0,1])$ and $\{h_c:\ c\in K\}$ is an equicontinuous subset of $C_0^\infty([0,1])$ for any compact set $K\subseteq \{z\in\C: \text{Re}(z)>0\}$.
\end{lemma}
\begin{proof}
From the hypothesis it follows that, for each $n\in\N$ there is $M_0^n>0$ such that, for each $c\in\C$ and $x\in [0,1]$, we have

$$|h_c(x)|\leq \int_0^x |g(t)| dt\leq M_0^n x^n.$$

\noindent The derivative satisfy $h_c'(x)=g(x)+cqx^{q-1}h_c(x)$. Inductively we get, for each $j\in\N$, polynomials $P_i^j$ of three variables, $0\leq i\leq j$ such that

$$h_c^{(j)}(x)= g^{(j-1)}(x)+\sum_{i=0}^{j-2}P_i^j(x^{-1},x^q,c)g^{(i)}(x)+P_{j-1}^{j}(x^{-1},x^q,c)h_c(x).$$

\noindent Hence we conclude $h_c(x)\in C_0^\infty([0,1])$ from Proposition \ref{description}. The continuity of each $P_i(x^{-1},x^q,c)$ with respect to $c$ yields that, for each $j,n\in\N_0$, there exist constants $M_j^n>0$ such that, for each $c\in K$, $x\in [0,1]$ 

$$|h_c^{(j)}(x)|\leq M_j^n x^n\leq M_j^n.$$

\end{proof}
\begin{theorem}
\label{wes}
For $p\in\R$, consider the operator $T_p:=M_{\omega_p}V:\ C_{0}^{\infty}([0,1])\to C_{0}^{\infty}([0,1])$, $f\mapsto x^{p}\int_0^xf(t)dt$,  (or $T_p:=VM_{\omega_p}:\ C_{0}^{\infty}([0,1])\to C_{0}^{\infty}([0,1])$, $f\mapsto \int_0^x t^{p}f(t)dt$). 

\begin{itemize}
\item[(i)] If $p \geq -1$ then $\sigma(T_p)=\emptyset$ and $\sigma^*(T_p)=\{0\}$.
\item[(ii)] If $p<-1$ then $\sigma(T_p)=\sigma_p(T_p)=\{\lambda\in\C:\ Re(\lambda)> 0\}$ and $\sigma^*(T_p)=\overline{\sigma(T_p)}$.
\end{itemize}
\end{theorem}

\begin{proof}
By Proposition \ref{esp2}, we only have to prove the statement for $T_p=M_{\omega_p}V$. Since $T_p$ is an (composition of) isomorphism, then $0\in \varrho(T_p)$ for each $p\in\R$. Let $\lambda\neq 0$, $T_p(f)=\lambda f$ for some $f\neq 0$ and only if

\begin{equation}
\label{ps}
V(f)(x)=\lambda x^{-p}f(x)
\end{equation}

\noindent For $p\neq -1$, since $f=DV(f)$ for every $f \in C_{0}^{\infty}([0,1])$, a solution of $\eqref{ps}$ is of the form $V(f)(x)=h_{\lambda,p}(x):=e^{\frac{x^{p+1}}{\lambda (p+1)}}$,  and for $p=-1$ we get $V(f)(x)=h_{\lambda,p}(x):=\omega_{\frac{1}{\lambda}}(x)=x^{\frac1\lambda}$.  
Hence $\lambda\in \sigma_p(T_p)$ if and only if $(h_{\lambda,p})' \in C_0([0,\infty])$, which is equivalent to $h_{\lambda,p} \in C_0([0,\infty])$.

\noindent We check first (i). For any $p\geq -1$, $h_{\lambda,p}\notin C_{0}^{\infty}([0,1])$, therefore. $\lambda I-T_p$ is injective for every $\lambda\in\C$.

 Let $p=-1$, which corresponds to the Ces\`aro operator $C(f):=T_{-1}(f)=\frac{1}{x}\int_0^xf(t)dt$. We see now that, for every $\lambda\in\C$, $\lambda I-C$ is surjective on $C_0^\infty([0,1])$. Let $g\in C_0^\infty([0,1])$. The function 
 
 \begin{equation}
 \label{resolvente}
\hat{g}(x):=\frac1\lambda x^{1/\lambda}\int_0^x g(t)t^{-\frac{1}{\lambda}}dt,
 \end{equation}
 
\noindent  is a solution of the differential equation
 
 $$\lambda y'-\frac1x y=g.$$
 
\noindent By putting $y=V(f)$ in the equation we get

 $$\lambda f-C(f)=g.$$
 
 \noindent Proposition \ref{isomorphism} and Corollary \ref{multipliers2} yield  $\hat{g}\in C_0^\infty([0,1])$, and thus $\lambda I-C$ is surjective. Since we already know that $\lambda I-C$ is injective, we conclude that  $\lambda I- C$ is an isomorphism, i.e. $\lambda\notin \sigma(C)$.  We have proved then $\sigma(C)=\emptyset$. Since $V$ is an isomorphism on $C_0^\infty([0,1])$ whose inverse is $D$, we have  
 $$R(\lambda, C)(g)=\frac{1}{\lambda}DM_{\omega_{1/\lambda}}V(M_{\omega_{-1/\lambda}}(f)).$$

   \noindent Let $K\subset \C\setminus\{0\}$ and $f\in C_0^\infty([0,1])$. We define
 
 $$
 B(K,f):=\left\{R(\lambda,C)f: \lambda\in K\right\}=\left\{\frac{1}{\lambda}DM_{\omega_{1/\lambda}}V(M_{\omega_{-1/\lambda}}(f)):\lambda \in K\right\}
 $$

 \noindent  Proposition \ref{eqex} (a) together with the fact that equicontinuous sets are equibounded, yields the boundedness of $B(K,f)$. Hence $\C\setminus\{0\}\subseteq \varrho^*(C)$, and consequently $\sigma^*(C)\subseteq \{0\}$. Let  $f\in C_0^\infty([0,1])$, such that $f\geq 0$ and $\int_0^1 f(t)dt=1$. For every $\lambda\in\C$ we have
 
 $$\frac{1}{\lambda}M_{\omega_{1/\lambda}}VM_{\omega_{-1/\lambda}}(f)(x)=\frac1\lambda\int_0^x \left(\frac{x}{t}\right)^\frac{1}{\lambda}f(t)dt.$$
 
 \noindent Hence, for $\delta_1\in C_0^\infty([0,1])'$ being the evaluation functional at 1, and $0<\lambda<1$, we get
 
 $$\left\langle \delta_1, \frac1\lambda M_{\omega_{-1/\lambda}}VM_{\omega_{1/\lambda}}(f)\right\rangle=\frac{1}{\lambda}\int_0^1\left(\frac{1}{t}\right)^\frac{1}{\lambda}f(t)dt\geq \frac1\lambda.$$
 
 \noindent From this we conclude $0\notin \varrho^*(C)$, i.e. $\sigma^*(C)=\{0\}$.

 For $p>-1$, we proceed analogously in order to get
 
 \begin{equation}
 \label{resp}
  R(\lambda,T_p)(g)=\frac1\lambda DM_{h_{\lambda,p}}VM_{h_{-\lambda,p}}(g),
  \end{equation}
  
 \noindent  where $h_{\lambda,p}(x):=e^{\frac1\lambda\frac{x^{p+1}}{p+1}}.$ The conclusion is obtained analogously to the case $p=-1$, by using Proposition \ref{eqex} (b).
 
 We see now (ii). Let $p<-1$. Now $h_{\lambda,p}(x)=e^{\frac{x^{p+1}}{\lambda (p+1)}}$ belongs to $C_0^\infty([0,1])$ if and only if $\text{Re}(\lambda)>0$.
 Hence $\sigma_p(T_p)=\{\lambda: \ \text{Re}(\lambda)>0\}$. 
 
 Let $\lambda\in \C\setminus\{0\}$ such that $\text{Re}(\lambda)\leq 0$. When exists, the resolvent $R(\lambda ,T_p)$ must satisfy again \eqref{resp}. For any $g\in C_0^\infty([0,1])$, we define
 
$$ \hat{g}_\lambda(x):=\frac{1}{\lambda}\int_0^x e^{\frac{x^{p+1}-t^{p+1}}{\lambda (p+1)}}g(t)dt.$$

\noindent We observe $v_\lambda(t):=e^{\frac{x^{p+1}-t^{p+1}}{\lambda (p+1)}}\in \mathcal{M}$ when $\lambda\neq 0$ and $\text{Re}(\lambda)\leq 0$.  Now we have
 $R(\lambda,T_p)(g)(x)=D(\hat{g}_\lambda(x))=DVM_{v_\lambda}(g)$ for every $\lambda\in \{z\in \C:\ \text{Re}z \leq 0\}$. Hence $\sigma(T_p)=\sigma_p(T_p)=\{\lambda\in\C: \ \text{Re}(\lambda)>0\}$.
 
 \noindent  Let $K\subseteq \{z\in\C:\ \text{Re}(z)<0\}$  compact. From Lemma \ref{tecnic} applied to $c=1/(\lambda (p+1))$ it follows that  $\{\hat{g}_\lambda: \lambda\in K$ is bounded in  $C_0^\infty([0,1])\}$.    Since $D$ is an isomorphism, then
 also $
\left\{R(\lambda,T_{p})f: \lambda\in K\right\}=\left \{ D(\hat{g}_\lambda(x)): \ \lambda\in K \subseteq C_0^\infty([0,1])\right \}$
  is bounded.  Hence we conclude $\{\lambda\in\C:\  \text{Re}(\lambda)<0\}\subseteq\varrho^*(T_p)$. Since $\overline{\sigma}(T_p)\subseteq \sigma^*(T_p)$
  we conclude $\sigma^*(T_p)= \{\lambda \in\C:\  \text{Re}(\lambda)\geq 0\}$.

\end{proof}
\begin{theorem}
\label{dif}
For $p\in\R$, let $T_p:=M_{\omega_p}D:\ C_{0}^{\infty}([0,1])\to C_{0}^{\infty}([0,1])$, $f(x)\mapsto x^{p}f'(x)dt$,  (or $T_p:=DM_{\omega_p}:\ C_{0}^{\infty}([0,1])\to C_{0}^{\infty}([0,1])$, $f(x)\mapsto  (x^{p}f(x))'dt$). 

\begin{itemize}
\item[(i)] If $p \leq 1$ then $\sigma(T_p)=\sigma^*(T_p)=\emptyset$.
\item[(ii)] If $p>1$ then $\sigma(T_p)=\sigma_p(T_p)=\{\lambda\in\C:\ Re(\lambda)> 0\}$ and $\sigma^*(T_p)=\overline{\sigma(T_p)}$.
\end{itemize}
\end{theorem}

\begin{proof}
Combining Proposition \ref{inverse} and Theorem \ref{wes} we obtain all the statements except $0\in \varrho^*(T_p)$ when $p\leq 1$. For $p<1$, the resolvent  $R(\lambda,T_p)$ can be directly computed:
$$R(\lambda,T_p)(f)=-x^p \int_0^x  e^{\frac{\lambda (x^{-p+1}-t^{-p+1}) }{-p+1}}f(t)dt=-M_{\omega_p}M_{h_{\lambda,-p}}V M_{h_{-\lambda,-p }}(f),$$

\noindent for $h_{\lambda,-p}(x):=e^{ \frac{\lambda x^{-p+1} }{-p+1} }$, $\omega_p(x):=x^p$. The equicontinuity of $\{R(\lambda,T_p)(f):\ \lambda\in K\}$ for each $f\in  C_{0}^{\infty}([0,1])$
and $K\subseteq \C$ compact follows from Proposition \ref{eqex} (b) and the fact that equicontinuous sets are equibounded. 

For $p=1$, we have 

$$R(\lambda,T_1)(f)=-x\int_0^x  \left(\frac{x}{t}\right)^\lambda f(t)dt=-M_{\omega_1}M_{\omega_\lambda}V M_{\omega_{-\lambda}}(f),$$

\noindent and we conclude in an analogous way using Proposition \ref{eqex} (a).
\end{proof}

We finish with a description of the spectrum of the differentiation operator in the one sided Schwartz class, and we observe that the composition of these operators with multiplication by monomials (or other powers of $x$) can be described with the same arguments.  

\begin{theorem}
Let consider the differentiaton operator $D: \ S([1,\infty))\to S([1,\infty))$, $f\mapsto f'$ and its inverse $I:\ S([1,\infty))\to S([1,\infty))$, $f\mapsto -\int_x^\infty f(t)dt$.
\begin{itemize}
\item[(i)] $\sigma(D)=\sigma_p(D)=\{\lambda\in\C: \ \text{Re}(\lambda)<0\}$ and $\sigma^*(D)=\overline{\sigma(D)}$.
\item[(ii)] $\sigma(I)=\sigma_p(I)=\{\lambda\in\C: \ \text{Re}(\lambda)<0\}$ and $\sigma^*(I)=\overline{\sigma(I)}$.
\end{itemize}
\end{theorem}
\begin{proof}
Observe that, by means of the isomorphism defined on Proposition \ref{description}, $D$ is equivalent (making and abuse of notation)  to $-M_{\omega_2} D: C_0^\infty([0,1])\to C_0^\infty([0,1])$, $f\mapsto -x^2f'(x)$.  Now (i) follows now  from Theorem \ref{dif} (ii). Statement (ii) is a consequence of Proposition \ref{inverse} and (i).
\end{proof}

\noindent {\bf Acknowledgements.} The author would like to thank J. Bonet for his careful reading of the preprint and the given suggestions which really improved the work. In particular, he observed and  provided the arguments showing that some results stated for particular operators could be written in a very general and abstract way. These ideas leaded to Subsection 1.2. Also Example 3 given in this subsection was obtained in a joint discussion. Also it is worth to mention that L. Frerick suggested the representation of the space of flat functions given in Theorem \ref{representation}  some time ago. This research was supported by PID2020-119457GB-I00 and GVA-AICO/2021/170.

\noindent {\bf Data Availability.} Data sharing not applicable to this article as no datasets were generated or analysed during the current study.

\end{document}